\documentclass[a4,12pt, leqno]{amsart}
\usepackage{amsmath,amsfonts,amsthm,latexsym,amssymb,mathrsfs}
\usepackage{bbm}

\def\R{\mathcal{R}}
\def\A{\mathcal{A}}
\def\X{\mathcal{X}}
\def\L{\mathcal{L}}
\def\S{\mathcal{S}}

\newtheorem{df}{Definition}[section]
\newtheorem{thm}[df]{Theorem}
\newtheorem{pro}[df]{Proposition}
\newtheorem{cor}[df]{Corollary}

\newtheorem{rema}[df] {Remark}
\newtheorem{lem}[df] {Lemma}

\begin{document}

\setcounter{page}{1}

\title[The $(b,c)$-inverse in rings and in  the Banach context]{The $(b, c)$-inverse in rings and  in the Banach context}

\author[E. Boasso and G. Kant\'un-Montiel]{Enrico Boasso and Gabriel Kant\'un-Montiel}
\maketitle

\begin{abstract}
In this article the $(b, c)$-inverse will be studied. Several equivalent conditions for the existence of the $(b,c)$-inverse in rings will be given. In particular, the conditions ensuring the existence of the $(b,c)$-inverse,  of the annihilator $(b,c)$-inverse and  of the hybrid $(b,c)$-inverse will be proved to be equivalent, provided $b$ and $c$ are regular elements in a unitary ring $\R$. In addition, the set of all $(b,c)$-invertible elements will be characterized and the reverse order law will be also studied. 
Moreover, the relationship between the $(b,c)$-inverse and the Bott-Duffin inverse will be considered. In the context of Banach algebras, integral,  series and limit representations will be given.
Finally the continuity of the $(b,c)$-inverse will be characterized. \par
\vskip.2truecm
\noindent Mathematics Subject Classification. 15A09, 16B99, 16U99, 46H05.\par
\vskip.2truecm
\noindent  Keywords. Unitary ring, $(b, c)$-inverse, Bott-Duffin $(p, q)$-inverse, Outer inverse, Group inverse, Banach algebra. 
\end{abstract}

\section{Introduction}

\noindent Given a semigroup $\S$ and elements $a$, $b$, $c\in\S$, the notion of a $(b, c)$-inverse of $a$ was introduced in \cite{D} (see \cite[Definition 1.3]{D} or section 2).
This inverse is a class of outer inverses which encompasses several well known  generalized inverses, such as the Drazin inverse and the  Moore-Penrose inverse (in the presence of an involution). 
Furthermore, the aforementioned inverse has been studied by several authors (see \cite{D, D2, G, CCW, KC}) and it is closely related to
other outer inverses such as the $a_{p, q}^{(2, l)}$ outer inverse in Banach algebras  and the image-kernel $(p, q)$-inverse in rings (see \cite{CX, G, MDG} and sections 2 and 7).\par

\indent The aim of this article is to study further properties of the $(b,c)$-inverse enlarging the underlying set, in particular, considering the contexts of unitary rings and Banach algebras.
In section 3, after having recalled some preliminary definitions and facts in section 2, further equivalent conditions that ensure the existence of the outer inverse  under consideration
will be given. From this result it will be derived that the  $(b, c)$-inverse, the annihilator $(b, c)$-inverse and the hybrid $(b, c)$-inverse  (see \cite{D}) coincide, when
$b$ and  $c$ are regular. In section 4 the set of all $(b, c)$-invertible elements in a unitary ring will be characterized and more equivalent characterizations of the 
$(b, c)$-inverse will be proved. In section 5 the relationship between the $(b, c)$-inverse and the  Bott-Duffin $(p, q)$-inverse (see \cite[Definition 3.2]{D}) will be studied. In fact, on the one hand,  
the latter inverse is a particular case of the former inverse (when $b$ and $c$ are idempotents), but on the other, it will be proved that $(b, c)$-invertible elements are Bott-Duffin $(p, q)$-invertible
for suitable $p$ and $q$. In section 6 the reverse order law for the $(b,c)$-inverse and the Bott-Duffin $(p, q)$-inverse will be studied. Finally, in section 7 several results concerning integral representations, convergent series, limits, approximation 
and continuity of the $(b, c)$-inverse will be presented. 

\section{Preliminary definitions and facts}

\noindent From now on $\R$ will denote a unitary ring with unity 1. Let $\R^{-1}$ (respectively $\R^\bullet$) be the set of
invertible elements (respectively of idempotents) of $\R$. Given $a \in \R$, 
the {\em image ideals} are defined by $a \R : = \{ ax: x \in \R \}$ and $\R a : = \{ xa : x \in \R \}$, and 
the {\em kernel ideals} by $a^{-1}(0) : = \{ x \in \R: ax=0 \}$ and $a_{-1}(0) : = \{ x \in \R: xa=0 \}$.

An element $a \in \R$ is said to be {\it regular}, if there exists $x \in \R$ such that $a=axa$.
The element $x$, which is not uniquely determined by $a$, will be said to be  a {\it generalized inverse}
or an {\it inner inverse} of $a$. The set of regular elements of $\R$ will be denoted by $\widehat{\R}$ and,
given $v\in\hat{\R}$, $v\{1\}$ will stand for the set of all inner inverses of $v$. In addition, 
if $y \in \R$ satisfies $yay=y$, then $y$  is said to be an {\it outer inverse} of $a$.
An element $z\in\R$ is said to be a {\it normalized generalized inverse} of $a$, if
$z$ is an inner inverse and an outer inverse of $a$. Note that if $w\in\R$ is an inner inverse of
$a$, then $w'=waw$ is a normalized generalized inverse of $a$.

Now the definition of the key notion of this article will be recalled.\par

\begin{df}[{\cite[Definition 1.3]{D}}]\label{def1}
Let $\R$ be a ring with unity and consider $b, c \in \R$. The element $a\in\R$
will be said to be $(b,c)$-invertible, if there exists $y \in \R$ such that the following
equations hold:\par
\noindent {\rm (i)} $y\in(b\R y)\cap (y\R c)$,\par
\noindent {\rm (ii)} $b=yab$, $c=cay$.
\end{df}

\indent If such inverse exists, then it is unique (see \cite[Theorem 2.1 (i)]{D}). Thus in what follows,
if the element $y$ in Definition \ref{def1} exists, then it will be denoted by $a^{-(b,\hbox{ }c)}$. In addition,
according to \cite[Theorem 2.1 (ii)]{D}, $a^{-(b,\hbox{ }c)}$ is an outer inverse of $a$.\par

\indent In the following remark some properties of $b$ and $c$ will be considered (see also \cite{G}).

\begin{rema}\label{rema3}\rm Let $\R$ be a ring with unity and let $a$, $b$, $c\in \R$. Recall that according to \cite[Proposition 6.1]{D},
necessary and sufficient for $a$ to be $(b, c)$-invertible is that there exists $y\in\R$ such that $yay=y$,
$y\R= b\R$ and $\R y=\R c$.\par
\noindent (i) Let $b'$, $c'\in\R$ be such that $b'\R= b\R$ and $\R c'= \R c$. Then, from \cite[Proposition 6.1]{D},
$a$ is  $(b, c)$-invertible if and only if $a$ is  $(b', c')$-invertible. Moreover, in this case $a^{-(b',\hbox{ }c')}=a^{-(b,\hbox{ }c)}$.\par
\noindent Next suppose that $a^{-(b,\hbox{ }c)}$ exists.\par
\noindent (ii) Let $z\in\R$. Recall that necessary and sufficient for $z$ to have a right inverse (respectively a left inverse) 
is that $z\R=\R$ (respectively $\R z=\R$). Consequently, applying again \cite[Proposition 6.1]{D}, necessary and sufficient for $a\in \R^{-1}$ is that $b$ is right invertible
and $c$ is left invertible.   In this case, $a^{-(b,\hbox{ }c)}=a^{-1}$.\par
\noindent (iii) The elements $b$ and $c$ are regular. In fact, according to Definition \ref{def1}, there are $g$, $h\in\R$
such that $bga^{-(b,\hbox{ }c)}=a^{-(b,\hbox{ }c)}=a^{-(b,\hbox{ }c)}hc$. Therefore,
\begin{align*}
b&=a^{-(b,\hbox{ }c)}ab=bga^{-(b,\hbox{ }c)}ab=bgb,\\
c&=caa^{-(b,\hbox{ }c)}=caa^{-(b,\hbox{ }c)}hc=chc.\\
\end{align*}
\noindent (iv) Let $g'=gbg$ and $h'=hch$. Then, $g'$ and $h'$ are normalized generalized inverses of 
$g$ and $h$, respectively. In addition, since according to Definition \ref{def1}, there is $z\in\R$ such that 
$a^{-(b,\hbox{ }c)}=bzc$, $a^{-(b,\hbox{ }c)}=bg'a^{-(b,\hbox{ }c)}$ and $a^{-(b,\hbox{ }c)}=a^{-(b,\hbox{ }c)}h'c$. In fact,
\begin{align*}
&bg'a^{-(b,\hbox{ }c)}=bg'bzc=bzc=a^{-(b,\hbox{ }c)}.\\
&a^{-(b,\hbox{ }c)}h'c=bzch'c=bzc=a^{-(b,\hbox{ }c)}.\\
\end{align*}

\noindent Therefore, in Definition \ref{def1}, not only $b$, $c\in\hat{\R}$, but also the elements
that satisfy condition (i) of Definition \ref{def1} ($g$ and $h$ in (iii)) can be chosen as normalized generalized inverses
of $b$ and $c$.
\end{rema}

\indent A particular case of the $(b,c )$-inverse is the Bott-Duffin inverse.

\begin{df}[{\cite[Definition 3.2]{D}}]\label{def15}Let $\R$ be a ring with unity and consider $p$, $q\in\R^\bullet$. The element
$a\in\R$ will be said to be Bott-Duffin $(p, q)$-invertible, if there exists $y\in \R$ such that \par
\noindent {\rm (i)} $y=py=yq$,\par
\noindent {\rm (ii)} $yap=p$ and $qay=q$.
\end{df}

\indent Clearly, given $p$, $q\in\R^\bullet$, the Bott-Duffin $(p, q)$-inverse is nothing but the $(b, c)$-inverse 
when $b$ and $c$ are idempotents. In addition, since there exists at most one $(b, c)$-inverse,
the Bott-Duffin inverse is unique, if it exists. According to what has been said, if $a\in\R$ is
Bott-Duffin  $(p,q)$-invertible, then the element $y$ in Definition \ref{def15} will be denoted by
$a^{-(p, \hbox{ }q)}$. On the other hand, it is worth noticing that the image-kernel $(p, q)$-inverse (see \cite{G, MDG}) 
is closely related to  the Bott-Duffin $(p, q)$-inverse. In fact, according to \cite[Proposition 3.4]{G},
this inverse coincides with the   Bott-Duffin $(p, 1-q)$-inverse.  \par

\indent In the article \cite{D} two other outer inverses which are related to the $(b,c)$-inverse were considered.
Next their definitions will be recalled.

\begin{df}[{\cite[Definitions 6.2-6.3]{D}}]\label{def10}Let $\R$ be a ring with unity and consider $a$, $b$, $c \in \R$. \par
\noindent \textnormal{(i).} The element $y\in\R$ will be said to be 
an annihilator $(b,c)$-inverse of $a$, if 
$$
y=yay, \hskip.3truecm y_{-1}(0)=b_{-1}(0), \hskip.3truecm y^{-1}(0)= c^{-1}(0).
$$
\noindent \textnormal{(ii).} The element $y\in\R$ will be said to be a hybrid $(b,c)$-inverse of $a$, if 
$$
y=yay, \hskip.3truecm y\R=b\R, \hskip.3truecm y^{-1}(0)= c^{-1}(0).
$$
\end{df}

\indent In the same conditions of Definition \ref{def10}, according to \cite[Theorem 6.4]{D}, there can be at most one annihilator (respectively hybrid) $(b,c)$-inverse of $a\in \R$.
Moreover, it is not difficult to prove that if $a$ has a $(b,c)$-inverse, then $a$ has a hybrid inverse, which in turn implies that $a$ has an annihilator inverse.\par

\indent The last generalized inverse  in the context of rings that it is necessary to recall for this article is the group inverse.
Let $\R$ be a ring with unity and consider $a\in\R$. The element $a$ will be said 
to be \it group invertible, \rm if there exists $b\in\R$ such that 
$$
a=aba, \hskip.3truecm  b=bab, \hskip.3truecm ab=ba.
$$
\noindent It is well known that the group inverse is unique, if it exists. In that case, the group inverse of $a\in\R$ will be denoted by
$a^\sharp$. Note that $aa^\sharp$, $a^\sharp a\in\R^\bullet$ and 
\begin{align*}
a\R&=aa^\sharp\R=a^\sharp a\R=a^\sharp \R,\\
 \R a&=\R a^\sharp a=\R aa^\sharp =\R a^\sharp,\\
a^{-1}(0)&=(a^\sharp a)^{-1}(0)=(aa^\sharp )^{-1}(0)=(a^\sharp)^{-1}(0),\\
a_{-1}(0)&=(aa^\sharp )_{-1}(0)=(a^\sharp a)_{-1}(0)=(a^\sharp)_{-1}(0),\\
\R&=a\R \oplus a^{-1}(0)=\R a\oplus a_{-1}(0).
\end{align*}

\indent In the Banach context, other generalized inverses need to be recalled. To this end, however,
some preparation is needed.\par

\indent From now on, $\A$ will denote a unitary Banach algebra with unit $1$. Given $a\in\A$, $\sigma (a)$ will stand for the spectrum of $a$ and when $\lambda\in\mathbb{C}$, $\lambda 1$ will be
written simply as $\lambda$. In addition, $\X$ will denote a Banach space and $\L (\X)$ the Banach algebra of all
linear and bounded maps defined on and with values in $\X$.  If $T\in\L (\X)$, then $N(T)$ and $R(T)$ will stand for 
the null space and the range of the operator $T$, respectively. In particular, when  $\A$ is a Banach algebra and $x\in \A$,  the operators $L_x\colon \A\to \A$ and $R_x\colon \A\to \A$ 
are the maps defined as follows: given $z\in \A$, $L_x(z)= xz$ and $R_x (z)= zx$, respectively. Note that  the identity operator defined on the 
Banach space $\X$ will be denoted by $I\in\L(\X)$.\par

\indent The following generalized inverse was introduced in the context
of Banach algebras in \cite{CX}. Although the approach in this latter article is different from the one in \cite{D}, the generalized inverse
introduced in \cite{CX} is related to the $(b,c)$-inverse (see  section 7). \par

\begin{df}[{\cite[Definition 2.10]{CX}}]\label{def21}Let $\A$ be a Banach algebra and consider $a\in\A$ and $p$, $q,\in\A^\bullet$. The element $y\in\A$
will be said to be the $(p, q, l)$-outer inverse of $a$, if the following identities hold:
$$
y=yay, \hskip.3truecm y\A=p\A, \hskip.3truecm y^{-1}(0)=q\A.
$$
In this case, this outer inverse will be denoted by $a^{(2,\hbox{ }l)}_{p,\hbox{ }q}$.
\end{df}

\indent Note that, in the same conditions of Definition \ref{def21}, the $(p,q,l)$-outer inverse coincides with the hybrid $(p, 1-q)$-inverse. \par

\indent One of the most studied generalized inverses is the $A^{(2)}_{T,\hbox{ }S}$ outer inverse. This generalized inverse was studied for matrices and for operators defined on Hilbert and on Banach spaces. 
Since in this article  Banach space operators will be considered, the definition of the aformentioned outer inverse will be given in the Banach frame. Recall that this inverse
is unique, when it exists (see \cite[Lemma 1]{LYZW}).\par  

\begin{df}\label{def30}Let $\X$ be a Banach space, $A\in\L (\X)$ and $T$ and $S$ two closed subspaces of $\X$. If there exists a necessarily unique operator $B\in\L (\X)$ such that
$B$ is an outer inverse of $A$ and $R(B)=T$ and $N(B)=S$, then $B$ will be said to be the $A^{(2)}_{T,\hbox{ }S}$ outer inverse of $A$. 
\end{df}
\indent According to \cite[Lemma 1]{LYZW}, necessary and sufficient for the existence of the $A^{(2)}_{T,\hbox{ }S}$ outer inverse of $A$ is that
$T$ and $S$ are complemented subspaces of $\X$, $A\mid_T^{A(T)}\colon T\to A(T)$ is invertible and $A(T)\oplus S=\X$. In particular, using this latter decomposition,
$A^{(2)}_{T,\hbox{ }S}$ is the following operator: $A^{(2)}_{T,\hbox{ }S}\mid_S=0$, $(A\mid_T^{A(T)})^{-1}=A^{(2)}_{T,\hbox{ }S}\mid^T_{A(T)}\colon A(T)\to T$.
To learn more properties of the $A^{(2)}_{T,\hbox{ }S}$ outer inverse in Banach spaces, see \cite{LYZW, YW}.

\section{Further equivalent conditions}

\noindent In this section new conditions that are equivalent to the ones in Definition \ref{def1} will be given.
First of all note that if $\R$ is a unitary ring and $b$, $c$, $y\in \R$ are such that $y\R=b\R$, then $y_{-1}(0)=b_{-1}(0)$. Similarly from $\R y=\R c$ it can be derived that 
$y^{-1}(0)=c^{-1}(0)$. These results will be used in the following theorem.  
\par

\begin{thm}\label{thm11}Let $\R$ be a unitary ring and consider $a\in \R$. Let $y\in\R$ be an outer inverse of $a$.
Then, the following statements are equivalent.\par
\noindent \textnormal{(i)} The element  $y$ is the $(b, c)$-inverse of $a$.\par
\noindent  \textnormal{(ii)}  $\R y=\R c$, $y\R\subseteq b\R$ and $y_{-1} (0)\subseteq b_{-1}(0)$.\par
\noindent \textnormal{(iii)} $y\R=b R$, $\R y\subseteq \R c$ and $y^{-1}(0)\subseteq c^{-1}(0)$.\par
\noindent\rm  \textnormal{(iv)} $\R y\subseteq \R c$, $y\R\subseteq b\R$,   $y_{-1} (0)\subseteq b_{-1}(0)$
and $y^{-1}(0)\subseteq c^{-1}(0)$.\par
\noindent In addition, if $b$, $c\in \hat{\R}$, then the following statements are equivalent to statement \textnormal{(i)}.\par
\noindent  \textnormal{(v)} $\R y=\R c$, $b\R\subseteq y\R$ and $b_{-1} (0)\subseteq y_{-1}(0)$.\par
\noindent  \textnormal{(vi)} $y\R=b R$, $\R c\subseteq \R y$ and $c^{-1}(0)\subseteq y^{-1}(0)$.\par
\noindent  \textnormal{(vii)} $\R y= \R c$,   $y_{-1} (0)= b_{-1}(0)$.\par
\noindent  \textnormal{(viii)} $\R y\subseteq \R c$,  $b\R\subseteq y\R$, $y^{-1}(0)\subseteq c^{-1}(0)$
and $b_{-1} (0)\subseteq y_{-1}(0)$.\par
\noindent  \textnormal{(ix)} $\R c\subseteq \R y$, $y\R\subseteq b\R$, $y_{-1} (0)\subseteq b_{-1}(0)$
and $c^{-1}(0)\subseteq y^{-1}(0)$.\par
\noindent  \textnormal{(x)}  $\R c\subseteq\R y$, $b\R\subseteq y\R$, $c^{-1}(0)\subseteq y^{-1}(0)$ and
$b_{-1} (0)\subseteq y_{-1}(0)$.\par
\noindent  \textnormal{(xi)}   $y\R=b R$, $y^{-1}(0)= c^{-1}(0)$.\par
\noindent  \textnormal{(xii)} $\R y\subseteq \R c$,  $y^{-1}(0)\subseteq c^{-1}(0)$ and $y_{-1} (0)= b_{-1}(0)$.\par
\noindent  \textnormal{(xiii)}  $\R c\subseteq\R y$, $c^{-1}(0)\subseteq y^{-1}(0)$ and $y_{-1} (0)= b_{-1}(0)$.\par
\noindent  \textnormal{(xiv)}  $b\R\subseteq y\R$, $b_{-1} (0)\subseteq y_{-1}(0)$ and $y^{-1} (0)= c^{-1}(0)$.\par
\noindent  \textnormal{(xv)} $y\R\subseteq b\R$, $y_{-1} (0)\subseteq b_{-1}(0)$ and $y^{-1} (0)= c^{-1}(0)$.\par
\noindent  \textnormal{(xvi)}  $y^{-1} (0)= c^{-1}(0)$ and $y_{-1} (0)= b_{-1}(0)$.\par

\end{thm}
\begin{proof} According to \cite[Proposition 6.1]{D} and the observation at the beginning of this section, it is clear that statement (i) implies all the other statements.\par
\indent On the other hand, to prove that statements (ii)-(iv) imply that $y$ is the $(b,c)$-inverse of $a$,
note that according to \cite[Proposition 3.1 (i)]{bb1} (respectively  \cite[Proposition 3.1 (ii)]{bb1}), if $y^{-1}(0)\subseteq c^{-1}(0)$ 
(respectively if $y_{-1}(0)\subseteq b_{-1}(0)$), then $\R c\subseteq \R y$
(respectively $b\R\subseteq y\R$).\par
\indent Suppose that $b$ and $c$ are regular. Then, according to \cite[Proposition 3.1 (iii)]{bb1}
(respectively \cite[Proposition 3.1 (iv)]{bb1}), if $c^{-1}(0)\subseteq y^{-1}(0)$ 
(respectively if  $b_{-1}(0)\subseteq y_{-1}(0)$), then $\R y \subseteq \R c$
(respectively $y\R\subseteq b\R$). To conclude the proof, apply these results and the ones
used to prove the equivalence among statements (i)-(iv).
\end{proof}

\indent Note that statements  (xi) and (xx) in Theorem \ref{thm11} correspond to the conditions of the
hybrid and the annihilator $(b, c)$-inverse of $a$, respectively. The following corollary can be 
easily deduced from the previous theorem.\par

\begin{cor}\label{coro12} Let $\R$ be a unitary ring and consider $a\in\R$ and $b$, $c\in\hat{\R}$. Then,
the following statements are equivalent.\par
\noindent \textnormal{(i)} The element $a$ has a $(b,c)$-inverse.\par
\noindent \textnormal{(ii)} The element $a$ has an annihilator $(b,c)$-inverse.\par
\noindent \textnormal{(iii)} The element $a$ has a hybrid $(b,c)$-inverse.\par
\noindent Furthermore, in this case the inverses in statements \textnormal{(i)-(iii)}
coincide.
\end{cor} 
\begin{proof} First apply the equivalence among statements (i), (xi) and (xx) in Theorem \ref{thm11}
and then apply \cite[Theorem 2.1]{D} and \cite[Theorem 6.4]{D}.
\end{proof}

\noindent Let $\R$ be a unitary ring and consider $a$, $b\in \R$ and $c\in\hat{\R}$ such that
the hybrid $(b,c)$-inverse of $a$ exists. Note that accoding to the proof of Theorem \ref{thm11},
the $(b,c)$-inverse of $a$ exists and it coincides with the hybrid $(b,c)$-inverse of $a$.
In particular, according to Remark \ref{rema3} (i), $b\in\hat{\R}$. 

\section{The set of $(b, c)$-invertible elements} 

\noindent The main objective of this section is to study the set of all $(b, c)$-invertible elements.
In first place a characterization of the $(b,c)$-inverses will be given.\par

\begin{thm}\label{thm4} Let $\R$ be a ring with unity and consider $b$, $c\in\hat{\R}$. Consider  $a\in\R$, $g\in b\{1\}$ and $h\in c\{1\}$.  
Then, the following conditions are equivalent.\par
\noindent {\rm (i)} The element $a$ is $(b,c)$-invertible.\par
\noindent {\rm (ii)} There exists  $z\in b\R c$ such that 
$$
bg=zhcabg, \,\,\,\,\, hc=hcabgz.
$$
Furthermore, in this case $z= a^{-(b,\hbox{ }c)}$.
\end{thm}
\begin{proof} Suppose that $a^{-(b,\hbox{ }c)}$ exists. According to Definition \ref{def1} (i), $a^{-(b,\hbox{ }c)}\in\ b\R c$.
Next note that $b\R c=bg \R hc$. 
In particular, $a^{-(b,\hbox{ }c)}= bg a^{-(b,\hbox{ }c)}=  a^{-(b,\hbox{ }c)}hc$. Thus, according  to Definition \ref{def1} (ii),
$$
bg=a^{-(b,\hbox{ }c)}abg=a^{-(b,\hbox{ }c)} hcabg,\,\,\, hc=hcaa^{-(b,\hbox{ }c)}=hcabg a^{-(b,\hbox{ }c)}.
$$
\indent To prove the converse, note that as before, since $z\in b\R c=bg\R hc$, $z=bgz=zhc$. In particular,
$z\in b\R z$ and $z\in z\R c$. In addition,
$$
b=bgb=zhcabgb=zab, \,\,\,\,  c=chc= chcabgz=caz.
$$
\noindent Therefore, $z$ satisfies the conditions of Definition \ref{def1}.
\end{proof}

\indent From Theorem \ref{thm4} new $(b,c)$-invertible elements can be constructed.
This will be done in the following corollaries.\par

\begin{cor}\label{cor13}Let $\R$ be a ring with unity and consider $a\in\R$, $b$, $c\in\hat{\R}$, $g\in b\{1\}$ and $h\in c\{1\}$.  Then, the following statements are equivalents.\par
\noindent {\rm (i)} $a$ is $(b,c)$-invertible.\par
\noindent {\rm (ii)}  $hca$ is $(b,c)$-invertible.\par
\noindent {\rm (iii)}  $abg$ is $(b,c)$-invertible.\par
\noindent {\rm (iv)}  $hcabg$ is $(b,c)$-invertible.\par
\noindent Furthermore, in this case, 
$$
a^{-(b,\hbox{ }c)}=(hca)^{-(b,\hbox{ }c)}=(abg)^{-(b,\hbox{ }c)}=(hcabg)^{-(b,\hbox{ }c)}.
$$
\end{cor}
\begin{proof} Note that 
$$
hcabg= hc(hca)bg= hc(abg)bg=hc(hcabg)bg.
$$
Now, necessary and sufficient for $a$ to be $(b,c)$-invertible is that there exists $z\in \R$ such that
Theorem \ref{thm4} (ii) holds. However, this condition is equivalent to the existence of the outer inverses in statements (ii)-(iv).
Moreover, in this case, all the outer inverses considered in statements (i)-(iv) coincide with $z$.
\end{proof}

\begin{cor}\label{cor5}Let $\R$ be a ring with unity and consider $b$, $c\in\hat{\R}$. Let $a\in\R$ be such that 
$a^{-(b,\hbox{ }c)}$ exists and consider $m\in hc\R (1-bg) + (1-hc)\R$, where $g\in b\{1\}$ and $h\in c\{1\}$. Then, $a+m$ is $(b,c)$-invertible. Moreover, $(a+m)^{-(b,\hbox{ }c)}=a^{-(b,\hbox{ }c)}$.
\end{cor}
\begin{proof}Note that $hc(a+m)bg=hcabg$. Then, as in the proof of Corollary \ref{cor13}, there exists $z\in\R$ that satisfies  Theorem \ref{thm4} (ii).
Now, this condition is also equivalent to the $(b, c)$-invertibility of $a+m$. Furthermore, both outer inverses under consideration coincide.
\end{proof}

\begin{cor}\label{cor14} Let $\R$ be a ring with unity and consider $a\in\R$ and $b$, $c\in\hat{\R}$
such that $a$ is $(b, c)$-invertible. Let $g\in b\{1\}$ and $h\in c\{1\}$. Then, the following statements hold.\par
\noindent {\rm (i)} $a^{-(b,\hbox{ }c)}+m $ is $(hc, bg)$-invertible, where $m\in bg\R (1-hc)+ (1-bg)\R$.\par
\noindent {\rm (ii)} $(a^{-(b,\hbox{ }c)}+m)^{-(hc, \hbox{ }bg)}= hcabg$; in particular, $(a^{-(b,\hbox{ }c)})^{-(hc, \hbox{ }bg)}= hcabg$. 
\end{cor}
\begin{proof} Note that since $hc$  and $bg$ are indempotents, $hc\in hc\{1\}$ and $bg\in bg\{1\}$.
In addition, since $a^{-(b,\hbox{ }c)}\in b\R c=bg\R hc$ (Theorem \ref{thm4}), 
$bg (a^{-(b,\hbox{ }c)}+m)hc=a^{-(b,\hbox{ }c)}$. In particular, applying Theorem \ref{thm4},
the $(hc, bg)$-inverse of $a^{-(b,\hbox{ }c)}+m $ exists and it coincides with $hcabg$. 
\end{proof}
\indent In the following theorem the set of all $(b,c)$-invertible elements will be characterized 
($b$, $c\in\hat{\R}$). To this end, however, first two sets need to be introduced. Let 
$$
\R^{-(b,\hbox{ }c)}:=\{x\in\R\colon x^{-(b,\hbox{ }c)} \hbox{ exists }\}
$$
and
$$
\R^{-1}_{\{hc, \hbox{ } bg\}}:=\{x\in hc\R bg\colon \hbox{ there exists } z\in b\R c \hbox{ such that }
bg=zx, \hbox{ }hc=xz\},
$$
\noindent where $g\in b\{1\}$ and $h\in c\{1\}$ are arbitrary elements.

\begin{thm}\label{thm6}Let $\R$ be a ring with unity and consider $b$, $c\in\hat{\R}$. Let $g\in b\{1\}$ and $h\in c\{1\}$.
Then,
$$
\R^{-(b,\hbox{ }c)}= \R^{-1}_{\{hc, \hbox{ } bg\}}+ hc\R (1-bg) + (1-hc)\R.
$$
\noindent Furthermore, if $x\in \R^{-1}_{\{hc, \hbox{ } bg\}}$ and $m\in hc\R (1-bg) + (1-hc)\R$, then $(x+m)^{-(b,\hbox{ }c)}= x^{-(b,\hbox{ }c)}\in b\R c$.
\end{thm}
\begin{proof}Let $x\in\R^{-1}_{\{hc, \hbox{ } bg\}}$ and consider $z\in  b\R c$  such that 
$bg=zx$ and $hc=xz$. According to Theorem \ref{thm4}, $x$ is $(b, c)$-invertible. In addition,
$x^{-(b,\hbox{ }c)}=z$. Now let $m\in hc\R (1-bg) + (1-hc)\R$. According to Corollary \ref{cor5},
$x+m$ is  $(b, c)$-invertible and $(x+m)^{-(b,\hbox{ }c)}= x^{-(b,\hbox{ }c)}=z$.\par

\indent On the other hand, let $a\in\R^{-(b,\hbox{ }c)}$. According to Theorem \ref{thm4}, $hcabg=x\in  \R^{-1}_{\{hc, \hbox{ } bg\}}$.
Note that $a= hcabg +m$, where $m= hca(1-bg) +(1-hc)a\in hc\R (1-bg) + (1-hc)\R$.
\end{proof}

\indent  From Theorem \ref{thm6} a new characterization of $(b, c)$-invertible elements can be derived.\par

\begin{cor}\label{cor7}Let $\R$ be a ring with unity and consider $b$, $c\in\hat{\R}$. Let $a\in\R$ and consider 
$g\in b\{1\}$ and $h\in c\{1\}$. Then, the following statements are equivalent.\par
\noindent {\rm (i)} The element $a$ is $(b, c)$-invertible.\par
\noindent {\rm (ii)} There exist $u$, $v\in\R$ such that $u\in  \R^{-1}_{\{hc, \hbox{ } bg\}}$, $hcvbg=0$ and $a=u+v$.
\end{cor}
\begin{proof}Apply Theorem \ref{thm6}.
\end{proof}

\indent To end this section, the relationship between $\R^{-1}_{\{hc, \hbox{ } bg\}}$ and the product will be studied.
Recall however first that given a ring with unity $\R$ and $p\in\R^\bullet$,
$p\R p$ is a ring with unit $p$. Moreover, if $t\in (p\R p)^{-1}$, then its inverse in $p\R p$
will be denoted by $t^{-1}_{p\R p}$. In particular, $tt^{-1}_{p\R p}=t^{-1}_{p\R p}t=p$.\par

\begin{thm}\label{thm400}Let $\R$ be a ring with unity and consider $b$, $c\in\hat{\R}$. Let $g\in b\{1\}$ and $h\in c\{1\}$.
The following statements hold.\par
\noindent {\rm (i)} $\R^{-1}_{\{hc, \hbox{ } bg\}}=(hc\R hc)^{-1}\R^{-1}_{\{hc, \hbox{ } bg\}}(bg\R bg)^{-1}$.\par
\noindent {\rm (ii)} Let $x\in \R^{-1}_{\{hc, \hbox{ } bg\}}$ and consider $u\in (bg\R bg)^{-1}$ and $v\in (hc\R hc)^{-1}$.
Let $m\in hc\R (1-bg)+ (1-hc)\R$. Then, $vxu+m\in \R^{-(b,\hbox{ }c)}$ and 
$$
(vxu)^{-(b,\hbox{ }c)}=(vxu +m)^{-(b,\hbox{ }c)}= u^{-1}_{bg\R bg} x^{-(b,\hbox{ }c)} v^{-1}_{hc\R hc}.
$$
\end{thm}
\begin{proof} Note that since $\R^{-1}_{\{hc, \hbox{ } bg\}}\subset hc\R bg$, 
$$
\R^{-1}_{\{hc, \hbox{ } bg\}}=hc \R^{-1}_{\{hc, \hbox{ } bg\}} bg\subseteq (hc\R hc)^{-1}\R^{-1}_{\{hc, \hbox{ } bg\}}(bg\R bg)^{-1}.
$$
\indent To prove the other inclusion, let $x\in \R^{-1}_{\{hc, \hbox{ } bg\}}$. Since $hcxbg=x$, according to Theorem \ref{thm4}, there exists $z\in b\R c=bg\R hc$ such that
$zx=bg$ and $xz=hc$. Consider $u\in (bg\R bg)^{-1}$ and $v\in (hc\R hc)^{-1}$ and set $x'=vxu$ and $z'=u^{-1}_{bg\R bg}x^{-(b,\hbox{ }c)} v^{-1}_{hc\R hc}$.
Note that $x'\in hc\R bg$ and $z'\in bg\R hc=b\R c$. Moreover, direct calculations proves that  $z'x'=bg$ and $x'z'=hc$. In particular,
the first statement holds.\par
\indent To conclude the proof, apply Corollary \ref{cor5}.
\end{proof}
\section{The Bott-Duffin inverse}

\noindent In this section the relationship between the $(b, c)$-inverse and the Boot-Duffin $(p, q)$-inverse will be studied.
Recall that the Bott-Duffin inverse is a particular case of the $(b, c)$-inverse. The next result will 
show that the these definitions are essentially the same. In fact, any $(b, c)$-inverse can be presented as a
Bott-Duffin inverse.

\begin{pro}\label{pro16} Let $\R$ be a ring with unity and consider $b$, $c\in\hat{\R}$. Let $a\in\R$ and consider 
$g\in b\{1\}$ and $h\in c\{1\}$. The following statements are equivalent.\par
\noindent {\rm (i)} The element $a$ is $(b, c)$-invertible.\par
\noindent {\rm (ii)} The element $a$ is Bott-Duffin $(bg, hc)$-invertible.\par
\noindent Furthermore, in this case 
$$
a^{-(b,\hbox{ }c)}=a^{-(bg,\hbox{ }hc)} \hbox{  and    }\R^{-(b,\hbox{ }c)}=\R^{-(bg,\hbox{ }hc)}.
$$
\end{pro}
\begin{proof} Apply Theorem \ref{thm4}.
\end{proof}

\indent Next the relationship between the Bott-Duffin inverse and the usual inverse will be studied in a 
particular case.
To this end, recall first the following well known result. Let $\R$ be a ring with unity and
consider $p\in\R^\bullet$ and $a\in \R$ such that $ap=pa$ (equivalently $a=pap +(1-p)a(1-p)$).
Then, necessary and sufficient for $a\in\R^{-1}$ is that  $pap\in (p\R p)^{-1}$ and $(1-p)a(1-p)\in ((1-p)\R (1-p))^{-1}$. 
Moreover, in this case, a direct calculation proves that 
if $a_1\in p\R p$ (respectively $a_2\in (1-p)\R(1-p)$)
is the inverse of $pap$ in $p\R p$ (respectively the inverse of  $(1-p)a(1-p)$ in $(1-p)\R (1-p)$), then $a^{-1}=a_1+a_2$.
In the following lemma,
this result will be extended to the case of two idempotents.\par

\begin{lem} \label{lem18} Let $\R$ be a ring with unity and consider $p$, $q\in \R^\bullet$.
Supose that $a\in\R$ is such that $ap=qa$. Then the following statements are equivalent.\par
\noindent {\rm (i)} The element $a$ is invertible.\par
\noindent {\rm (ii)} There is $z\in \R$ such that $zq=pz$ and 
\begin{align*}
&pzqap=p, \,\,\,\, \, \, \,    (1-p)z(1-q)a(1-p)=1-p, \\
&qapzq=q,  \,\,\,\, \, \, \,   (1-q)a(1-p)z(1-q)=1-q.\\
\end{align*}
\noindent Furthermore, in this case $z=a^{-1}$.
\end{lem}
\begin{proof}Note that the condition $ap=qa$ is equivalent to  $a= qap +(1-q)a(1-p)$,
which in turn is equivalent to $(1-q)ap=0=qa(1-p)$. Similarly, given $z\in\R$, necessary and
sufficient for $zq=pz$ is that $z=pzq+(1-p)z(1-q)$,  equivalently, $(1-p)zq=0=pz(1-q)$.\par

\indent The statement of the Lemma can be proved by direct computation. The details are left to the reader.
\end{proof}

\indent Next a relation between the Bott-Duffin inverse and the usual inverse will be established.\par

\begin{thm}\label{thm19}Let $\R$ be a ring with unity and consider $p$, $q\in\R^\bullet$. 
Let $a\in\R$ be such that $ap=qa$. Then, the following statements are equivalent.\par
\noindent {\rm (i)} The element $a\in \R^{-1}$.\par
\noindent {\rm (ii)} Both the Bott-Duffin $(p, q)$-inverse of $a$ and the Bott-Duffin $(1-p, 1-q)$-inverse of $a$ exist.\par
\noindent Furthermore, in this case $a^{-1}= a^{-(p,\hbox{ } q)} + a^{-(1-p,\hbox{ } 1-q)}$.
\end{thm}
\begin{proof} Since $ap=qa$, $a=a_1+ a_2$, where $a_1= qap\in q\R p$ and $a_2=(1-q)a(1-p)\in (1-q)\R (1-p)$. 
According to Lemma \ref{lem18}, necessary and sufficient for $a\in\R ^{-1}$ is that there exists $z=z_1+z_2$,
$z_1= pzq\in p\R q$ and $z_2=(1-p)z(1-q)\in (1-p)\R (1-q)$, such that
\begin{align*}
&z_1a_1=p, \,\,\,\, \, \, \,\,\,\,\,\,\,\,\,\,\, a_1z_1=q,\\
&z_2a_2=1-p, \,\,\,\, \, \, \,  a_2z_2=1-q.\\
\end{align*}
Moreover, in this case $z=a^{-1}$. However, acording to Theorem \ref{thm4}, these equations are equivalent
to the fact that $a^{-(p,\hbox{ } q)}$ and $a^{-(1-p,\hbox{ } 1-q)}$ exist. \par

\indent To prove the converse, let $z= a^{-(p,\hbox{ } q)} + a^{-(1-p,\hbox{ } 1-q)}$. Clearly, $zq=pz$.
In addition, since $pzq=a^{-(p,\hbox{ } q)}$ and $(1-p)z(1-q)= a^{-(1-p,\hbox{ } 1-q)}$, according to Lemma \ref{lem18}, $a$ in invertible.
\end{proof}

\indent In the following corollary the relationship between the $(b, c)$-inverse and the usual inverse will be studied.\par

\begin{cor}\label{cor20}Let $\R$ be a ring with unity and consider $b$, $c\in \hat{\R}$.
Let $g\in b\{1\}$ and $h\in c\{1\}$. Then, if $a\in\R$ is such that $abg=hca$, the following statements are equivalent.\par
\noindent {\rm (i)} The element $a\in \R^{-1}$.\par
\noindent {\rm (ii)} $a$ is $(b, c)$-invertible and the Bott-Duffin $(1-bg, 1-hc)$-inverse of $a$ exists.\par
\noindent Furthermore, in this case $a^{-1}= a^{-(b,\hbox{ } c)} + a^{-(1-bg,\hbox{ } 1-hc)}$.
\end{cor}
\begin {proof} Define $p=bg$ and $q=hc$. Apply then Theorem \ref{thm19} and Proposition \ref{pro16}.
\end{proof}
\section{The reverse order law}

\noindent The main objective of this section is to study the reverse order law. In first place the notion of reverse order law for the $(b,c)$-inverse will be 
introduced.\par

\begin{df}\label{def8} Let $\R$ be a ring with unity and consider $b_i$, $c_i\in\hat{\R}$ ($i=1$, $2$). 
Let $g_i\in b_i\{1\}$ and $h_i\in c_i\{1\}$ ($i=1$, $2$) and suppose that $h_2c_2=b_1g_1$. 
Let $a_i\in \R$ and suppose that $a^{-(b_i, \hbox{ }c_i)}$ exist  ($i=1$, $2$). It will be said that
$a_1a_2$ satisfies the reverse order law, if  $a_1a_2$ is $(b_2,c_1)$-invertible and 
$$
(a_1a_2)^{-(b_2,\hbox{ } c_1)}=a_2^{-(b_2,\hbox{ } c_2)}a_1^{-(b_1,\hbox{ } c_1)}.
$$
\end{df}

\indent Next the reverse order law will be studied.

\begin{thm}\label{thm9}Let $\R$ be a ring with unity and consider $b_i$, $c_i\in\hat{\R}$ ($i=1$, $2$). 
Let $g_i\in b_i\{1\}$ and $h_i\in c_i\{1\}$ ($i=1$, $2$) and suppose that $h_2c_2=b_1g_1$. 
Let $a_i\in \R$ and suppose that $a^{-(b_i, \hbox{ }c_i)}$ exist  ($i=1$, $2$). Then, the following statements are equivalent.\par
\noindent $\mathrm{(i)}$ $h_1c_1a_1(1-b_1g_1)a_2b_2g_2=0$.\par
\noindent $\mathrm{(ii)}$ The element $a_1a_2$ satisfies the reverse order law.
\end{thm}
\begin{proof}Let $a_i^{-(b_i, \hbox{ }c_i)}\in b_i \R c_i$ ($i=1$, $2$). Recall that according to Theorem \ref{thm4}, for $i=1$, $2$,
$$
b_ig_i=a_i^{-(b_i, \hbox{ }c_i)}h_ic_ia_ib_ig_i,\,\,\, \,\,\, h_ic_i=h_ic_ia_ib_ig_i a_i^{-(b_i, \hbox{ }c_i)}.
$$
\noindent Since $h_2c_2=b_1g_1$, it is not difficult to prove that
$$
h_1c_1a_1a_2b_2g_2=h_1c_1a_1b_1g_1a_2b_2g_2 + h_1c_1a_1(1-b_1g_1)a_2b_2g_2.
$$

\indent Suppose that $h_1c_1a_1(1-b_1g_1)a_2b_2g_2=0$.
Thus, $a_2^{-(b_2, \hbox{ }c_2)}a_1^{-(b_1, \hbox{ }c_1)}\in b_2\R c_1$ and, according to Theorem \ref{thm4},
\begin{align*}
a_2^{-(b_2, \hbox{ }c_2)}a_1^{-(b_1, \hbox{ }c_1)}h_1c_1a_1a_2b_2g_2&=a_2^{-(b_2, \hbox{ }c_2)}a_1^{-(b_1, \hbox{ }c_1)}h_1c_1a_1b_1g_1h_2c_2a_2b_2g_2\\
&=a_2^{-(b_2, \hbox{ }c_2)}b_1g_1h_2c_2a_2b_2g_2=a_2^{-(b_2, \hbox{ }c_2)}h_2c_2a_2b_2g_2\\
&=b_2g_2,
\end{align*}
and
\begin{align*}
h_1c_1a_1a_2b_2g_2a_2^{-(b_2, \hbox{ }c_2)}a_1^{-(b_1, \hbox{ }c_1)}&=h_1c_1a_1b_1g_1h_2c_2a_2b_2g_2a_2^{-(b_2, \hbox{ }c_2)}a_1^{-(b_1, \hbox{ }c_1)}\\
&=h_1c_1a_1b_1g_1h_2c_2a_1^{-(b_1, \hbox{ }c_1)}=h_1c_1a_1b_1g_1a_1^{-(b_1, \hbox{ }c_1)}\\
&=h_1c_1.
\end{align*}

Therefore, according to Theorem \ref{thm4}, $a_1a_2$ is $(b_2, c_1)$-invertible and $a_2^{-(b_2,\hbox{ } c_2)}a_1^{-(b_1,\hbox{ } c_1)}= (a_1a_2)^{-(b_2,\hbox{ } c_1)}$.\par

\indent To prove the converse, suppose that $a_1a_2$ is $(b_2, c_1)$-invertible and $a_2^{-(b_2,\hbox{ } c_2)}a_1^{-(b_1,\hbox{ } c_1)}= (a_1a_2)^{-(b_2,\hbox{ } c_1)}$.
In particular, according to Theorem \ref{thm4},
\begin{align*}
b_2g_2&=a_2^{-(b_2, \hbox{ }c_2)}a_1^{-(b_1, \hbox{ }c_1)}h_1c_1a_1a_2b_2g_2\\
              &=a_2^{-(b_2, \hbox{ }c_2)}a_1^{-(b_1, \hbox{ }c_1)}h_1c_1a_1b_1g_1a_2b_2g_2  +a_2^{-(b_2, \hbox{ }c_2)}a_1^{-(b_1, \hbox{ }c_1)}h_1c_1a_1(1-b_1g_1)a_2b_2g_2\\
              &=b_2g_2 +a_2^{-(b_2, \hbox{ }c_2)}a_1^{-(b_1, \hbox{ }c_1)}h_1c_1a_1(1-b_1g_1)a_2b_2g_2.\\
\end{align*}
As a result, 
$$
a_2^{-(b_2, \hbox{ }c_2)}a_1^{-(b_1, \hbox{ }c_1)}h_1c_1a_1(1-b_1g_1)a_2b_2g_2=0.
$$
Now, multiplying the above equation on the left side by $h_2c_2a_2b_2g_2$,  
$$
0=h_2c_2a_1^{-(b_1, \hbox{ }c_1)}h_1c_1a_1(1-b_1g_1)a_2b_2g_2=a_1^{-(b_1, \hbox{ }c_1)}h_1c_1a_1(1-b_1g_1)a_2b_2g_2.
$$
Finally, multiplying again the above equation on the left side by $h_1c_1a_1b_1g_1$,
$$
h_1c_1a_1(1-b_1g_1)a_2b_2g_2=0.
$$
\end{proof}

\indent Since the Bott-Duffin inverse is a particular case of the $(b ,c)$-inverse, the following corollary can be directly derived from
Theorem \ref{thm9}.\par

\begin{cor}\label{cor17}Let $\R$ be a ring with unity and consider $p$, $q$, $r\in\R^\bullet$.  
Let $a_i\in \R$, $i=1$, $2$, and suppose that $a_1$ is Bott-Duffin $(p, q)$-invertible and
$a_2$ is Bott-Duffin $(r, p)$-Bott Duffin invertible. 
Then, the following statements are equivalent.\par
\noindent $\mathrm{(i)}$ $qa_1(1-p)a_2r=0$.\par
\noindent $\mathrm{(ii)}$ The element $a_1a_2$ is Bott-Duffin $(r, q)$-invertible and
$$
(a_1a_2)^{-(r,\hbox{ } q)}=a_2^{-(r,\hbox{ } p)}a_1^{-(p,\hbox{ } q)}.
$$

\end{cor}
\begin{proof}Apply Theorem \ref{thm9} and use $p$ (respectively $q$, $r$) as a generalized inverse of the idempotent $p$ (respectively $q$, $r$).
\end{proof}
\bibliographystyle{amsplain}

\section{Properties of the $(b, c)$-inverse in Banach algebras}

\noindent In this section several results concerning representations, limits, approximation and continuity of the 
$(b, c)$-inverse will be studied in the context of Banach algebras. In first place a lemma will be presented. \par

\begin{lem}\label{lem22}Let $\A$ be a Banach algebra and consider $a\in\A$ and $b$, $c\in \hat{\A}$.
Let $g\in b\{ 1\}$ and $h\in c\{ 1\}$. Then, the following statements are equivalent.\par
\noindent {\rm (i)} The element $a$ is $(b, c)$-invertible.\par
\noindent {\rm (ii)} The outer inverse $a^{(2,\hbox{ } l)}_{bg, \hbox{ } 1-hc}$ exists.\par
\noindent Furthermore, in this case $a^{-(b,\hbox{ }c)}$ and $a^{(2,\hbox{ }l)}_{bg,\hbox{ } 1-hc}$ coincide.
\end{lem}
\begin{proof}Recall that according to Proposition \ref{pro16}, statement (i) is equivalent to the existence
of the Bott-Duffin $(bg, hc)$-inverse of $a$, moreover,  this two inverses coincide, i.e., $a^{-(b,\hbox{ }c)}=a^{-(bg,\hbox{ }hc)}$.
In addition, according to Corollary \ref{coro12}, statement (i) is equivalent to the existence of the
hybrid  $(bg, hc)$-inverse of $a$ and this latter inverse coincides with $a^{-(b,\hbox{ }c)}=a^{-(bg,\hbox{ }hc)}$.
However, according to Definition \ref{def21}, the hybrid $(bg, hc)$-inverse of $a$ coincides with  $a^{(2,\hbox{ } l)}_{bg, \hbox{ }1-hc}$.
\end{proof}

\indent In order to prove the main results of this section, some facts need to be recalled first.\par

\begin{rema}\label{rem28}\rm Let $\A$ be a Banach algebra and consider $a\in\A$ and $b$, $c\in\hat{\A}$ such that 
$a^{-(b,\hbox{ }c)}$ exists. \par
\noindent (i) Since $a^{-(b,\hbox{ }c)}\in b\A c$ (Theorem \ref{thm4}), 
$a^{-(b,\hbox{ }c)}\A\subseteq b\A$ and $c^{-1}(0)\subseteq (a^{-(b,\hbox{ }c)})^{-1}(0)$. 
On the other hand, since according again to Theorem \ref{thm4}
$$
bg=a^{-(b,\hbox{ }c)}hcabg,\hskip.5truecm  hc= hcabga^{-(b,\hbox{ }c)},
$$
($g\in b\{ 1\}$  and $h\in c\{ 1\}$), $b\A\subseteq a^{-(b,\hbox{ }c)}\A$  and $(a^{-(b,\hbox{ }c)})^{-1}(0)\subseteq c^{-1}(0)$.
Thus, $a^{-(b,\hbox{ }c)}\A= b\A$ and $(a^{-(b,\hbox{ }c)})^{-1}(0)=c^{-1}(0)$.\par
\noindent Let $\A_{op}$ be the Banach algebra, which as Banach space coincides with $\A$, but whose product is the opposite
of the one in $\A$, i.e., $x\cdot_{op} y=yx$ ($x$, $y\in\A$). \par
\noindent (ii) Note that the $(b, c)$-inverse of $a\in \A$ exists if and only if the $(c, b)$-inverse of $a\in\A_{op}$ exits.
Moreover, in this case both inverses coincide.\par
\noindent (iii) Let $v\in \A$ such that $v\A=b \A$ and $v^{-1}(0)= c^{-1}(0)$. Then, according to Lemma \ref{lem22} and \cite[Theorem 4.1]{CX},
$av$ and $va$ are group invertible and $a^{-(b,\hbox{ }c)}=(va)^\sharp v=v(av)^\sharp$. In particular,
$$
va\A=v\A=b\A=bg\A, \hskip.5truecm (av)^{-1}(0)= v^{-1}(0)=c^{-1}(0)=(hc)^{-1}(0).
$$
Now an easy calculation proves the following fact: given $p$, $q\in \A^\bullet$, $p\A=q\A$ in $\A$ if and only if
$p^{-1}(0)=q^{-1}(0)$ in $\A_{op}$. Consequently, applying this result to $p=va$ and $q=bg$,
it is not dificult to prove that the identity $v \A=b \A$ in $\A$ is equivalent to the identity $v^{-1}(0)=b^{-1}(0)$  in $\A_{op}$.\par
\noindent Similarly, the following statement holds: $p^{-1}(0)=q^{-1}(0)$ in $\A$ if and only if $p\A=q\A$ in $\A_{op}$. 
Therefore, applying this latter statement to $p=av$ and $q=hc$, it is not difficult to prove that $v^{-1}(0)= c^{-1}(0)$ in $\A$
is equivalent to $v\A=c\A$ in $\A_{op}$.
\end{rema}

In the following theorem an integral representation of the $(b, c)$-inverse will be derived from the corresponding representation
of the outer inverse  $a^{(2,\hbox{ }l)}_{p, \hbox{ }q}$.

\begin{thm}\label{theo23}Let $\A$ be a Banach algebra and consider $a\in\A$ and $b$, $c\in\hat{\A}$ such that 
$a^{-(b,\hbox{ }c)}$ exists. Consider $v\in\A$ such that $v\A=b\A$ and $v^{-1}(0)=c^{-1}(0)$ and suppose that for every $\lambda\in\sigma (av)$, {\rm Re}$(\lambda )\ge 0$.
Then 
$$
a^{-(b,\hbox{ }c)}=\int_0^\infty ve^{-(av)t}dt=\int_0^\infty e^{-(va)t}vdt.
$$
\end{thm}
\begin{proof}Let $g\in b\{ 1\}$ and $h\in c\{ 1\}$. According to Lemma \ref{lem22}, $a^{(2,\hbox{ } l)}_{bg, \hbox{ } 1-hc}$ exits and 
$a^{-(b,\hbox{ }c)}=a^{(2,\hbox{ }l)}_{bg,\hbox{ } 1-hc}$. Next note that $bg\A=b\A$ and $c^{-1}(0)=(1-hc)\A$. Then apply \cite[Theorem 4.9]{CX} to deduce that
$$
a^{-(b,\hbox{ }c)}=\int_0^\infty ve^{-(av)t}dt.
$$
\indent To prove the second equality, first note that according to  statement (ii) of Remark \ref{rem28}, the  $(c, b)$-inverse of $a\in\A_{op}$
exists and coincide with $a^{-(b,\hbox{ }c)}$. In addition, according to statement (iii) of Remark \ref{rem28}, $v\in\A_{op}$ is such that  $v\cdot_{op}\A=c\cdot_{op}\A$ and
$v^{-1}(0)=b^{-1}(0)$ (in $\A_{op}$). Moreover, since $\sigma (a\cdot_{op}v)=\sigma (va)$ and $\sigma (va)\setminus\{0\}= \sigma (av)\setminus\{0\}$ (\cite[Chapter 1, Section 5, Proposition 3]{BD}),
$\lambda\in\sigma (av)$ is such that {\rm Re}$(\lambda )\ge 0$ if and only if
$\lambda\in\sigma (a\cdot_{op} v)=\sigma (va)$ is such that {\rm Re}$(\lambda )\ge 0$. Therefore, the second integral representation can be easily deduced from the identity that
has already been proved.
\end{proof}

\indent Next the $(b, c)$-inverse will be represented by means of a  convergent series. However, to this end,
it is necessary to recall some facts. \par

\indent Let $\A$ be a Banach algebra and consider $a\in\A$ and $b$, $c\in\hat{\A}$.  According to \cite[Theorem 4.1]{CX} 
and Lemma \ref{lem22}, if $a^{-(b,\hbox{ }c)}$ exists and $v\in\A$ is such that $v\A=b\A$ and $v^{-1}(0)=c^{-1}(0)$,
then $va$ is group invertible and $a^{-(b,\hbox{ }c)}=(va)^\sharp v$. Set $p_{va}=va(va)^\sharp$.
Note that since $va$ is group invertible, then $va=p_{va}va=vap_{va}=p_{va}vap_{va}$. In addition, $va\in (p_{va}\A p_{va})^{-1}$ and 
the inverse of $va$ in $p_{va}\A p_{va}$ is $(va)^\sharp\in p_{va}\A p_{va}$ ($p_{va}\A p_{va}$ is a Banach algebra with unit $p_{va}$).
Recall also that according to  \cite[Theorem 4.1]{CX} 
and Lemma \ref{lem22}, $av$ is group invertible, $a^{-(b,\hbox{ }c)}=v(av)^\sharp$, and if $p_{av}=av(av)^\sharp\in\A^\bullet$,
then $av=p_{av}av=avp_{av}=p_{av}avp_{av}$. \par

\begin{thm}\label{theo24}Let $\A$ be a Banach algebra and consider $a\in\A$ and $b$, $c\in\hat{\A}$ such that 
$a^{-(b,\hbox{ }c)}$ exists. Let $v\in\A$ such that $v\A=b\A$ and $v^{-1}(0)=c^{-1}(0)$ and consider $\beta\in\mathbb{C}\setminus\{0\}$
such that $\parallel p_{va}-\beta va\parallel <1$. Then
$$
a^{-(b,\hbox{ }c)}=\beta \sum_{n=0}^\infty (1-\beta va)^nv=\beta\sum_{n=0}^\infty v(1-\beta av)^n.
$$ 
\end{thm}
\begin{proof}Since $\parallel p_{va}-\beta va\parallel <1$, 
$$
(va)^\sharp=\beta \sum_{n=0}^\infty (p_{va}-\beta va)^n.
$$
\indent On the other hand, since $(1-\beta va)= p_{va}- \beta va + 1-p_{va}$, it is not difficult to prove that 
$(1-\beta va)^n= (p_{va}-\beta va)^n + 1-p_{va}$ ($n\in\mathbb{N}$). Thus, given $k\in\mathbb{N}$,
$$
\sum_{n=0}^k (1-\beta va)^n = \sum_{n=0}^k (p_{va}-\beta va)^n + (k+1)(1-p_{va}).
$$
\indent Now let $g\in b\{ 1\}$ and $h\in c\{ 1\}$. Note that according to \cite[Theorem 4.1]{CX} and Theorem \ref{thm4}

\begin{align*}
p_{va}&=(va)^\sharp va=a^{-(b,\hbox{ }c)}a= a^{-(b,\hbox{ }c)}hcabg +a^{-(b,\hbox{ }c)}hca(1-bg)\\
&= bg+a^{-(b,\hbox{ }c)}hca(1-bg).\\
\end{align*}
\noindent Since $v\in b\A=bg\A$, there is $z\in\A$ such that $v=bgz$. 
Thus,
$$
p_{va}v=bgv+a^{-(b,\hbox{ }c)}hca(1-bg)v=v,
$$
equiavalently, $(1-p_{va})v=0$.\par

\indent Consequently, given $k\in\mathbb{N}$,
$$
\beta \sum_{n=0}^k (1-\beta va)^nv=\beta (\sum_{n=0}^k (1-\beta va)^n)v= \beta\sum_{n=0}^k (p_{va}-\beta va)^nv.
$$

\indent Therefore, according to  \cite[Theorem 4.1]{CX},
$$
\beta \sum_{n=0}^\infty (1-\beta va)^nv= \beta\sum_{n=0}^\infty (p_{va}-\beta va)^nv=(va)^\sharp v=a^{-(b,\hbox{ }c)}.
$$

\indent To prove that $a^{-(b,\hbox{ }c)}$ coincides with the second serie, proceed as in the proof of Theorem \ref{theo23}.
In particular, recall that the $(c, b)$-inverse of $a\in\A_{op}$ exists and coincides with $a^{-(b,\hbox{ }c)}$. Moreover, 
$v\in\A_{op}$ is such that  $v\cdot_{op}\A=c\cdot_{op}\A$ and $v^{-1}(0)=b^{-1}(0)$ (in $\A_{op}$). Thus, to derive the proof of the
second representation from what has been already proved, it is enough to show that  
$\parallel p_{av}-\beta av\parallel=\parallel p_{va}-\beta va\parallel$.\par

\indent To this end, recall that since $va$ is group invertible, $\A=p_{va}\A\oplus (1-p_{va})\A$. Next consider the operator $L_{va}\colon \A\to \A$. 
Note that  (i) $L_{va}\mid_{(1-p_{va})\A}=0$, (ii) since $L_{va}\mid_{p_{va}\A}$ is invertible ($(L_{va}\mid_{p_{va}\A})^{-1}=L_{(va)^\sharp}\mid_{p_{va}\A}$), 
$0\notin\sigma (L_{va}\mid_{p_{va}\A})$, (iii) since $\sigma (va)=\sigma (L_{va})$ (\cite[Chapter 1, Section 5, Proposition 4 (ii)]{BD}),
$\sigma (va)=\{0\}\cup \sigma (L_{va}\mid_{p_{va}\A})$.\par

\indent On the other hand, according to \cite[Theorem 4.1]{CX} and Lemma \ref{lem22}, $av$ is group invertible. In particular, since $p_{av}=av(av)^\sharp\in\A^\bullet$, 
$\A=p_{av}\A\oplus (1- p_{av})\A$. In addition, using similar arguments to the ones in the previous paragraph, it is not difficult to prove that $\sigma (av)=\{0\}\cup \sigma (L_{av}\mid_{p_{av}\A})$
and $0\notin \sigma (L_{av}\mid_{p_{av}\A})$.\par
\indent Now since $\sigma (av)\setminus\{0\}=\sigma (va)\setminus\{0\}$ (\cite[Chapter 1, Section 5, Proposition 3]{BD}) and $0\notin  (\sigma (L_{av}\mid_{p_{av}\A})\cup \sigma (L_{va}\mid_{p_{va}\A}))$,
 $\sigma (L_{av}\mid_{p_{av}\A})= \sigma (L_{va}\mid_{p_{va}\A})$.\par

\indent Moreover, since $L_{p_{va}-\beta va}(p_{va}\A)\subseteq p_{va}\A$ and $L_{p_{va}-\beta va}(1-p_{va}\A)=0$, 
$\sigma(L_{p_{va}-\beta va})=\sigma(L_{p_{va}-\beta va}\mid_{p_{va}\A})\cup\{0\}$. Note that, according to the Functional
Calculus Theorem applied to the  polynomial $P(X)=1-\beta X\in\mathbb{C}[X]$ and the operator 
$L_{p_{va}-\beta va}\mid_{p_{va}\A}=(I-\beta L_{va})\mid_{p_{va}\A}$,  $\sigma(L_{p_{va}-\beta va}\mid_{p_{va}\A})=1-\beta\sigma(L_{va}\mid_{p_{va}\A})$.
However, using similar arguments it is possible to prove the following facts: $\sigma(L_{p_{av}-\beta av})=\sigma(L_{p_{av}-\beta av}\mid_{p_{av}\A})\cup\{0\}$
and $\sigma(L_{p_{av}-\beta av}\mid_{p_{av}\A})=1-\beta\sigma(L_{av}\mid_{p_{av}\A})$. As a result, 
$\sigma(L_{p_{va}-\beta va})=\sigma(L_{p_{av}-\beta av})$.\par

\indent Now, according to \cite[Chapter 1, Section 5, Theorem 8]{BD} and \cite[Chapter 1, Section 5, Proposition 4 (ii)]{BD},

\begin{align*}
\parallel p_{av}-\beta av\parallel&=\hbox{ max}\{\mid \lambda\mid\colon \lambda\in \sigma( p_{av}-\beta av)\}=\hbox{ max}\{\mid \lambda\mid\colon \lambda\in \sigma( L_{p_{av}-\beta av)}\}\\
&=\hbox{ max}\{\mid \lambda\mid\colon \lambda\in \sigma( L_{p_{va}-\beta va)}\}=\hbox{ max}\{\mid \lambda\mid\colon \lambda\in \sigma( p_{va}-\beta va)\}\\
&=\parallel p_{va}-\beta va\parallel\\
\end{align*}
\end{proof}

\indent Now due to \cite[Theorem 4.7]{CX} the $(b, c)$-inverse will be presented as a limit.\par

\begin{thm}\label{thm27}Let $\A$ be a Banach algebra and consider $a\in \A$ and $b$, $c\in\hat{\A}$ such that $a^{-(b,\hbox{ }c)}$ exists.
Let $v\in\A$ such that $v \A=b\A$ and $v^{-1}(0)=c^{-1}(0)$. Then 
$$
a^{-(b,\hbox{ }c)}=\lim_{\substack{\text{$\lambda\to 0$}\\\text{$\lambda\notin \sigma(-av)$}}} v(\lambda+av)^{-1}=
\lim_{\substack{\text{$\lambda\to 0$}\\\text{$\lambda\notin \sigma(-va)$}}} (\lambda+va)^{-1}v.
$$  
\end{thm}
\begin{proof}To prove the first identity, recall that according to Lemma \ref{lem22}, $a^{-(b,\hbox{ }c)}=a^{(2,\hbox{ } l)}_{bg,\hbox{ }1- hc}$.
In addition, given $g\in b\{ 1\}$ and $h\in c\{ 1\}$, according to the hypotheses,  $v\A=bg\A$ and $v^{-1}(0)=(1-hc)\A$. Thus, according to \cite[Theorem 4.1]{CX}, $av$ is group invertible.
In particular, acording to \cite[Theorem 4]{K}, $0$ is an isolated point of $\sigma (-av)$. Then, there is an open set $U\subset\mathbb{C}$
such that $0\in U$ and $U\cap \sigma (-av)=\{0\}$. To conclude the proof, apply \cite[Theorem 4.7]{CX}.\par

\indent In order to prove the remaining limit, proceed as in the proof of Theorem \ref{theo23}.
In particular, recall that $(c, b)$-inverse of $a\in\A_{op}$ exists and coincide with $a^{-(b,\hbox{ }c)}$. Moreover, 
$v\in\A_{op}$ is such that  $v\cdot_{op}\A=c\cdot_{op}\A$ and $v^{-1}(0)=b^{-1}(0)$ (in $\A_{op}$). In addition,
according to \cite[Chapter 1, Section 5, Proposition 3]{BD}, $\sigma (va)\setminus\{0\}=\sigma (av)\setminus\{0\}$. Therefore, 
the second identity can be derived from the first.
\end{proof} 

\indent Under the same hypothesis of Theorem \ref{thm27}, to establish a bound for $\parallel a^{-(b,\hbox{ }c)}-v(\lambda+av)^{-1}\parallel$,
it is necessary to note the following facts.\par

\begin{rema}\label{rem32}\rm Let $\A$ be a Banach algebra and consider $x$, $y\in\A$ two invertible elements. Then,
$$
x^{-1}-y^{-1}=y^{-1}(y-x)x^{-1}=(y^{-1}-x^{-1})(y-x)x^{-1} +x^{-1}(y-x)x^{-1}.
$$
Thus,
$$
\parallel x^{-1}-y^{-1}\parallel\le \parallel y^{-1}-x^{-1}\parallel\parallel y-x\parallel\parallel x^{-1}\parallel +\parallel x^{-1}\parallel^2\parallel y-x\parallel.
$$
\noindent In particular, when $\parallel y-x\parallel<\frac{1}{\parallel x^{-1}\parallel}$,
$$
\parallel x^{-1}-y^{-1}\parallel\le\frac{\parallel x^{-1}\parallel^2\parallel y-x\parallel}{1-\parallel y-x\parallel\parallel x^{-1}\parallel}.
$$
\end{rema}

\begin{rema}\label{rem29}\rm Let $\A$ be a Banach algebra and consider $a\in \A$ and $b$, $c\in\hat{\A}$ such that $a^{-(b,\hbox{ }c)}$ exists.
Let $g\in b\{ 1\}$, $h\in c\{ 1\}$  and  $v\in\A$ such that $v \A=b\A$ and $v^{-1}(0)=c^{-1}(0)$. Since according to Lemma \ref{lem22}, $a^{-(b,\hbox{ }c)}=a^{(2,\hbox{ } l)}_{bg,\hbox{ }1- hc}$,
due to \cite[Theorem 3.2]{CX}, $\A= ab\A\oplus c^{-1}(0)$ and $a^{-1}(0)\cap b\A=0$ (since $b$, $c\in\hat{\A}$, $b\A=bg\A$ and $(1-hc)\A=c^{-1}(0)$). In addition, according to
the proof of \cite[Theorem 4.1]{CX}, $(av)^{-1}(0)=v^{-1}(0)=c^{-1}(0)$. Now, since $av$ is group invertible (\cite[Theorem 4.1]{CX}),
$$
\A=av\A \oplus (av)^{-1}(0)= ab\A\oplus c^{-1}(0).
$$ 
\noindent Note that $av\A= av(av)^\sharp\A$ and   $(av)^{-1}(0)=(1-av(av)^\sharp)\A$ are two closed and complemented
subspace of $\A$.\par

\indent Consider the maps $L_a$, $L_v\colon \A\to \A$. Since  $a^{-1}(0)\cap v\A=0$, $L_a\mid_{v\A}^{av\A}\colon v\A\to av\A$ is an isomorphism.
Moreover, since $(av)^{-1}(0)=v^{-1}(0)$, $vav\A=v\A$ and $L_v\mid_{av\A}^{v\A}\colon av\A\to v\A$ is an isomorphism. Then 
$L_{av}\mid_{av\A}=L_a\mid_{v\A}^{av\A}L_v\mid_{av\A}^{v\A}\colon av\A\to av\A$  and 
$$
L_{(av)^\sharp}\mid_{av\A}= (L_{av}\mid_{av\A})^{-1}= (L_v\mid_{av\A}^{v\A})^{-1}(L_a\mid_{v\A}^{av\A})^{-1}
$$
\noindent (note that if $v=0$ or $a=0$, then $a^{-(b,\hbox{ }c)}=0$ and $\parallel a^{-(b,\hbox{ }c)}-v(\lambda+av)^{-1}\parallel=0$, i.e.,
no bound needs to be studied). In addition, note that since $a^{-(b,\hbox{ }c)}=a^{(2,\hbox{ } l)}_{bg,\hbox{ }1- hc}$,
$L_{a^{-(b,\hbox{ }c)}}=L_{a^{(2,\hbox{ } l)}_{bg,\hbox{ }1- hc}}=(L_a)^{(2)}_{b\A, \hbox{ }c^{-1}(0)}$.\par

\indent  Now consider the operator $H\in\L (\A)$ defined as follows: 
$$
H\mid_{(1-bg)\A}=0,\hskip.5truecm H\mid_{v\A}^{av\A}=(L_v\mid_{av\A}^{v\A})^{-1}\colon v\A\to av\A
$$
($bg\A=b\A= v\A$, $\A= bg\A\oplus (1-bg)\A$). Then, $L_{(av)^\sharp}=HL_{a^{-(b,\hbox{ }c)}}$.\par

\indent In fact, $N(L_{(av)^\sharp})=(av)^{-1}(0)=c^{-1}(0)=N((L_a)^{(2)}_{b\A, \hbox{ }c^{-1}(0)})=N(L_{a^{-(b,\hbox{ }c)}})$.
Then, 
$$
L_{(av)^\sharp}\mid_{(av)^{-1}(0)}=(HL_{a^{-(b,\hbox{ }c)}})\mid_{(av)^{-1}(0)}=0.  
$$
In addition, according to \cite[Remark 1]{LYZW} (see the paragraph that follows Definition \ref{def30}),
$L_{a^{-(b,\hbox{ }c)}}\mid_{av\A}^{v\A}=(L_a)^{(2)}_{b\A, \hbox{ }c^{-1}(0)}\mid_{av\A}^{v\A}=(L_a\mid_{v\A}^{va\A})^{-1}$.
Thus,  
$$
(HL_{a^{-(b,\hbox{ }c)}})\mid_{av\A}=H\mid_{v\A}^{av\A}L_{a^{-(b,\hbox{ }c)}}\mid_{av\A}^{v\A}=L_{(av)^\sharp}\mid_{av\A}.
$$
Since $\A=av\A \oplus (av)^{-1}(0)$, $L_{(av)^\sharp}=HL_{a^{-(b,\hbox{ }c)}}$.\par

\indent  Therefore, $(av)^\sharp=H(a^{-(b,\hbox{ }c)})$ and $\parallel (av)^\sharp\parallel\le \parallel H\parallel\parallel a^{-(b,\hbox{ }c)}\parallel$.
\end{rema}
 
\indent In the following theorem, a bound for  $\parallel a^{-(b,\hbox{ }c)}-v(\lambda+av)^{-1}\parallel$ (Theorem \ref{thm27}) will be given.
First three particular cases will be considered. Under the same hypotheses of Theorem \ref{thm27}, note that if  $a^{-(b,\hbox{ }c)}=0$,
then according to Definition \ref{def1}, $b=0=c$, which in turn implies that $v=0$. In particular, $v(\lambda+av)^{-1}=0$.
Note in particular that if $a=0$, then $a^{-(b,\hbox{ }c)}=0$. 
In addition, if the operator $H\in\L (\A )$ defined in Remark \ref{rem29} is such that $H=0$, then $v=0$, which implies that
$b=0=c$ and then,  according to Definition \ref{def1}, also $a^{-(b,\hbox{ }c)}=0$. As a result, when $\parallel a\parallel\parallel a^{-(b,\hbox{ }c)}\parallel\parallel H\parallel=0$,
$\parallel a^{-(b,\hbox{ }c)}-v(\lambda+av)^{-1}\parallel=0$ and no bound is necessary. \par

\begin{thm}\label{thm31}Let $\A$ be a Banach algebra and consider $a\in \A$ and $b$, $c\in\hat{\A}$ such that $a^{-(b,\hbox{ }c)}$ exists.
Let $v\in\A$ such that $v \A=b\A$ and $v^{-1}(0)=c^{-1}(0)$. Suppose that $\parallel a\parallel\parallel a^{-(b,\hbox{ }c)}\parallel\parallel H\parallel\neq 0$ and consider $\lambda\in \mathbb{C}\setminus \sigma(-av)$ such that 
$|\lambda| < \frac{1}{\parallel a\parallel\parallel a^{-(b,\hbox{ }c)}\parallel^2\parallel H\parallel}$, where $H\in\L (\A )$ is the operator defined in Remark \ref{rem29}. Then,
$$
\parallel a^{-(b,\hbox{ }c)}-v(\lambda+av)^{-1}\parallel \le  \frac{|\lambda |\parallel v\parallel\parallel a\parallel\parallel a^{-(b,\hbox{ }c)} \parallel^3\parallel H\parallel^2}{1-|\lambda |\parallel a\parallel\parallel a^{-(b,\hbox{ }c)}\parallel^2\parallel H\parallel}.
$$
\end{thm}
\begin{proof}According to Lemma \ref{lem22}  and  \cite[Theorem 4.1]{CX}, $a^{-(b,\hbox{ }c)}=a^{(2,\hbox{ } l)}_{bg,\hbox{ }1- hc}$,
$av$ is group invertible and $a^{-(b,\hbox{ }c)}=v(av)^\sharp$. In addition, since $\lambda\in \mathbb{C}\setminus \sigma(-av)$ and $0\in\sigma (-av)$ (\cite[Theorem 4]{K}),
$\lambda\neq 0$ and $\lambda+ av$ is invertible in $\A$. Moreover, since $\lambda+ av= (\lambda p_{av} + av) +\lambda (1-p_{av})$ and $av=p_{av}avp_{av}$,
$(\lambda+ av)^{-1}= (\lambda p_{av} + av)^{-1}_{p_{av}\A p_{av}} +\lambda^{-1}(1-p_{av})$, where 
$p_{av}= av(av)^\sharp\in\A^\bullet$ and  $(\lambda p_{av} + av)^{-1}_{p_{av\A p_{av}}}$ is the inverse of  $\lambda p_{av} + av$ in $p_{av} \A p_{av}$. 
\par

\indent Now note that according to Theorem \ref{thm4}, if $g\in b\{ 1\}$ and $h\in c\{ 1\}$, then
$$
av(av)^\sharp=aa^{-(b,\hbox{ }c)}= hcabg a^{-(b,\hbox{ }c)} +(1-hc) abg a^{-(b,\hbox{ }c)}= hc + (1-hc) abg a^{-(b,\hbox{ }c)}.
$$
\noindent However, since $v^{-1}(0)=c^{-1}(0)= (hc)^{-1}(0)$, $vav(av)^\sharp= vhc=v$, equivalently 
$v(1-p_{av})=0$. Consequently,
$$
a^{-(b,\hbox{ }c)}-v(\lambda+av)^{-1}=v((av)^\sharp-(\lambda +av)^{-1})=v((av)^\sharp- (\lambda p_{av} +av)^{-1}_{ p_{av}\A p_{av}}).
$$
\indent In particular,
$$
\parallel a^{-(b,\hbox{ }c)}-v(\lambda+av)^{-1}\parallel \le \parallel v\parallel \parallel (av)^\sharp- ( \lambda p_{av} +av)^{-1}_{ p_{av}\A p_{av}}\parallel.
$$
Note that if $(av)^\sharp=0$, since $a^{-(b,\hbox{ }c)}=v(av)^\sharp$, then $a^{-(b,\hbox{ }c)}=0$, which is impossible.
In particular, $p_{av}\neq 0$. 
Now, since $p_{av}=aa^{-(b,\hbox{ }c)}$, according to Remark \ref{rem29}, 
$|\lambda| < \frac{1}{\parallel a\parallel\parallel a^{-(b,\hbox{ }c)}\parallel^2\parallel H\parallel}\le\frac{1}{\parallel p_{av}\parallel\parallel (av)^\sharp\parallel}$.
Furthermore, since $(av)^\sharp$ is the inverse of $av$ in $p_{av}\A p_{av}$,  
according to Remark \ref{rem32},
$$
\parallel a^{-(b,\hbox{ }c)}-v(\lambda+av)^{-1}\parallel \le \parallel v\parallel \frac{|\lambda |\parallel p_{av}\parallel\parallel (av)^\sharp\parallel^2}{1-|\lambda |\parallel p_{av}\parallel\parallel (av)^\sharp\parallel}.
$$
Therefore, since $|\lambda |<\frac{1}{\parallel a\parallel\parallel a^{-(b,\hbox{ }c)}\parallel^2\parallel H\parallel}$, according to Remark \ref{rem29},
$$
\parallel a^{-(b,\hbox{ }c)}-v(\lambda+av)^{-1}\parallel \le  \frac{|\lambda |\parallel v\parallel\parallel a\parallel\parallel a^{-(b,\hbox{ }c)} \parallel^3\parallel H\parallel^2}{1-|\lambda |\parallel a\parallel\parallel a^{-(b,\hbox{ }c)}\parallel^2\parallel H\parallel}.
$$
\end{proof}

\indent To establish a bound for $\parallel a^{-(b,\hbox{ }c)}-(\lambda+va)^{-1}v\parallel$ (Theorem \ref{thm27}), it is necessary to note some facts first.\par

\begin{rema}\label{rema79}\rm Let $\A$ be a Banach algebra and consider $a\in \A$ and $b$, $c\in\hat{\A}$ such that $a^{-(b,\hbox{ }c)}$ exists.
Let $v\in\A$ be such that $v \A=b\A$ and $v^{-1}(0)=c^{-1}(0)$. Then, as in the proof of Theorem \ref{theo23},
the $(c, b)$-inverse of $a\in\A_{op}$ exists and coincide with $a^{-(b,\hbox{ }c)}$. Moreover, 
$v\in\A_{op}$ is such that  $v\cdot_{op}\A=c\cdot_{op}\A$ and $v^{-1}(0)=b^{-1}(0)$ (in $\A_{op}$). Thus, since $a^{-(b,\hbox{ }c)}-(\lambda+va)^{-1}v=
 a^{-(b,\hbox{ }c)}-v\cdot_{op}(\lambda+a\cdot_{op}v)^{-1}$, to give a bound for  $\parallel a^{-(b,\hbox{ }c)}-(\lambda+va)^{-1}v\parallel$,
it is enough to apply Theorem \ref{thm31} with the operator $\tilde{H}\in\L( \A)$, which is the map defined in Remark \ref{rem29}
that corresponds to the $(c, b)$-inverse of $a$ in $\A_{op}$. Thus, $\tilde{H}\in\L( \A)$ is defined as follows:
$$
\tilde{H}\mid_{(1-hc)\A}=0, \hskip.5truecm \tilde{H}\mid_{\A v}^{\A va}=(R_v\mid_{\A va}^{\A v})^{-1}\colon \A v\to \A va.
$$
\noindent Note that  when $\parallel a\parallel\parallel a^{-(b,\hbox{ }c)}\parallel\parallel \tilde{H}\parallel=0$,
$a^{-(b,\hbox{ }c)}=0$, $v=0$ and  $\parallel a^{-(b,\hbox{ }c)}-(\lambda+va)^{-1}v\parallel=0$, so that no bound is needed.
\end{rema}

\begin{cor}\label{coro799}Let $\A$ be a Banach algebra and consider $a\in \A$ and $b$, $c\in\hat{\A}$ such that $a^{-(b,\hbox{ }c)}$ exists.
Let $v\in\A$ such that $v \A=b\A$ and $v^{-1}(0)=c^{-1}(0)$. Suppose that  $\parallel a\parallel\parallel a^{-(b,\hbox{ }c)}\parallel\parallel \tilde{H}\parallel\neq 0$ 
and consider $\lambda\in \mathbb{C}\setminus \sigma(-va)$ such that 
$|\lambda| < \frac{1}{\parallel a\parallel\parallel a^{-(b,\hbox{ }c)}\parallel^2\parallel \tilde{H}\parallel}$, where $\tilde{H}\in\L (\A )$ is the operator defined in Remark \ref{rema79}. Then,
$$
\parallel a^{-(b,\hbox{ }c)}-(\lambda+va)^{-1}v\parallel \le  \frac{|\lambda |\parallel v\parallel\parallel a\parallel\parallel a^{-(b,\hbox{ }c)} \parallel^3\parallel \tilde{H}\parallel^2}{1-|\lambda |\parallel a\parallel\parallel a^{-(b,\hbox{ }c)}\parallel^2\parallel \tilde{H}\parallel}.
$$
\end{cor}
\begin{proof}Apply Theorem \ref{thm31} and Remark \ref{rema79}.
\end{proof}
\indent Next the continuity of the $(b, c)$-inverse inverse will be studied.
In first place, a technical lemma will be presented.\par

\begin{lem}\label{lem25}Let $\A$ be a Banch algebra and consider $a$, $b\in\A$ and $p$, $q$, $r$, $s\in \A^\bullet$ such that
$a^{-(p,\hbox{ }q)}$ and $b^{-(r,\hbox{ }s)}$ exist. Then
\begin{align*}
b^{-(r,\hbox{ }s)}-a^{-(p,\hbox{ }q)}&=b^{-(r,\hbox{ }s)}(s-q)(1-aa^{-(p,\hbox{ }q)})-  (1-b^{-(r,\hbox{ }s)}b)(r-p)a^{-(p,\hbox{ }q)}\\
&\hskip.5truecm + b^{-(r,\hbox{ }s)}(a-b)a^{-(p,\hbox{ }q)}.\\
\end{align*}
\end{lem}
\begin{proof}Recall that according to Theorem \ref{thm4},  $a^{-(p,\hbox{ }q)}\in p\A q$ and $b^{-(r,\hbox{ }s)}br=r$. Thus,

\begin{align*}
b^{-(r,\hbox{ }s)}ba^{-(p,\hbox{ }q)}-a^{-(p,\hbox{ }q)}&= -(1-b^{-(r,\hbox{ }s)}b)pa^{-(p,\hbox{ }q)}\\
\hskip.5truecm &=[(1-b^{-(r,\hbox{ }s)}b)r-(1-b^{-(r,\hbox{ }s)}b)p]a^{-(p,\hbox{ }q)}\\
\hskip.5truecm & =(1-b^{-(r,\hbox{ }s)}b)(r-p)a^{-(p,\hbox{ }q)}.
\end{align*}

\indent In addition, according again to Theorem \ref{thm4}, $b^{-(r,\hbox{ }s)}\in r\A s$ and $qaa^{-(p,\hbox{ }q)}=q$.
In particular,
\begin{align*}
b^{-(r,\hbox{ }s)}-b^{-(r,\hbox{ }s)}aa^{-(p,\hbox{ }q)}&=b^{-(r,\hbox{ }s)}s(1-aa^{-(p,\hbox{ }q)})\\
\hskip.5truecm &=b^{-(r,\hbox{ }s)}[s(1-aa^{-(p,\hbox{ }q)})-q(1-aa^{-(p,\hbox{ }q)})]\\
\hskip.5truecm &=b^{-(r,\hbox{ }s)}(s-q)(1-aa^{-(p,\hbox{ }q)})\\
\end{align*}
\indent Therefore,

\begin{align*}
b^{-(r,\hbox{ }s)}-a^{-(p,\hbox{ }q)}&=b^{-(r,\hbox{ }s)}(s-q)(1-aa^{-(p,\hbox{ }q)}) +b^{-(r,\hbox{ }s)}aa^{-(p,\hbox{ }q)}\\
\hskip.5truecm &- (1-b^{-(r,\hbox{ }s)}b)(r-p)a^{-(p,\hbox{ }q)}-b^{-(r,\hbox{ }s)}ba^{-(p,\hbox{ }q)}\\
\hskip.5truecm &=b^{-(r,\hbox{ }s)}(s-q)(1-aa^{-(p,\hbox{ }q)})-  (1-b^{-(r,\hbox{ }s)}b)(r-p)a^{-(p,\hbox{ }q)}\\
\hskip.5truecm &+ b^{-(r,\hbox{ }s)}(a-b)a^{-(p,\hbox{ }q)}.\\
\end{align*}
\end{proof}

\indent In the following Corollary the $(b, c)$-inverse case will be considered.\par

\begin{cor}\label{cor28000}Let $\A$ be a Banch algebra and consider  $a_i\in\A$ and $b_i$, $c_i\in \A^\bullet$  such that
$a_i^{-(b_i,\hbox{ }c_i)}$  exist ($i=1, 2$). Consider $g_i\in b_i\{1\}$ and $h_i\in c_i\{ 1\}$ ($i=1, 2$). Then
\begin{align*}
a_2^{-(b_2,\hbox{ }c_2)}-&a_1^{-(b_1,\hbox{ }c_1)}=a_2^{-(b_2,\hbox{ }c_2)}(h_2c_2-h_1c_1)(1-a_1a_1^{-(b_1,\hbox{ }c_1)})\\
&-  (1-a_2^{-(b_2,\hbox{ }c_2)}a_2)(b_2g_2-b_1g_1)a_1^{-(b_1,\hbox{ }c_1)}
 + a_2^{-(b_2,\hbox{ }c_2)}(a_1-a_2)a_1^{-(b_1,\hbox{ }c_1)}.\\
\end{align*}
\end{cor}
\begin{proof} According to Proposition \ref{pro16}, 
$$a^{-(b_1,\hbox{ }c_1)}=a^{-(b_1g_1,\hbox{ }h_1c_1)}, \hskip.3truecm a_2^{-(b_2,\hbox{ }c_2)}= a_2^{-(b_2g_2,\hbox{ }h_2c_2)}.$$

\noindent Now apply Lemma \ref{lem25} for $a=a_1$, $b=a_2$, $p=b_1g_1$, $q=h_1c_1$, $r=b_2g_2$ and $s=h_2c_2$.
\end{proof}
\indent Now the continuity of the $(b, c)$-inverse will be characterized.\par

\begin{thm}\label{theo26}Let $\A$ be a Banach algebra and consider $a\in \A$, $b$, $c\in\hat{\A}$, $g\in b\{1\}$ and $h\in c\{ 1\}$
such that $a^{-(b,\hbox{ }c)}$ exists. Supposse that there are sequences $(a_n)_{n\ge 1}\subset \A$,
$(b_n)_{n\ge 1}\subset\hat{\A}$, $(g_n)_{n\ge 1}\subset\A$, $(c_n)_{n\ge 1}\subset\hat{\A}$ and $(h_n)_{n\ge 1}\subset\A$ such that
for each $n\in\mathbb{N}$, $g_n\in b_n\{1\}$, $h_n\in c_n\{1\}$, $a_n^{-(b_n,\hbox{ }c_n)}$ exists and $(a_n)_{n\ge 1}$ (respectively
$(b_n g_n)_{\ge 1}$,  $(h_nc_n)_{n\ge 1}$) converges to $a$ (respectively to $bg$, $hc$).  Then, the following statements are equivalent.\par
\noindent {\rm (i)} There exists $k\in\mathbb{R}$ such that $\parallel  a_n^{-(b_n,\hbox{ }c_n)}\parallel \le k$.\par
\noindent {\rm (ii)} The sequence $a_n^{-(b_n,\hbox{ }c_n)}$ converges to $a^{-(b,\hbox{ }c)}$.
\end{thm}
\begin{proof}It is enough to prove that statement (i) implies statement (ii). To this end, apply Corollary  \ref{cor28000}.
\end{proof}

\vskip.5truecm
\noindent Enrico Boasso\par
\noindent E-mail: enrico\_odisseo@yahoo.it
\vskip.5truecm
\noindent  Gabriel Kant\'un-Montiel\par
\noindent E-mail: gkantun@fcfm.buap.mx

\end{document}